\documentclass[12pt,reqno]{amsart}
\usepackage{amsmath}
\usepackage{amssymb}
\usepackage{bbold}
\usepackage[all,cmtip]{xy}

\DeclareMathAlphabet{\mathpzc}{OT1}{pzc}{m}{it}
\newcommand{\ncom}{\newcommand}

\ncom{\rar}{\rightarrow}
\ncom{\imply}{\Rightarrow}
\ncom{\lrar}{\longrightarrow}
\ncom{\into}{\hookrightarrow}
\ncom{\onto}{\twoheadrightarrow}
\ncom{\ov}{\overline}
\ncom{\m}{\mbox}
\ncom{\sta}{\stackrel}
\ncom{\invlim}{\varprojlim}
\ncom{\xhat}{\widehat}

\ncom{\vspc}{\vspace{3mm}}
\ncom{\End}{{\cE}nd}
\ncom{\tensor}{\otimes}

\ncom{\al}{\alpha}
\ncom{\cHom}{{\mathcal Hom}}

\ncom{\A}{{\mathbb A}}
\ncom{\comx}{{\mathbb C}}
\ncom{\E}{{\mathbb E}}
\ncom{\F}{{\mathbb F}}
\ncom{\G}{{\mathbb G}}
\ncom{\K}{{\mathbb K}}
\ncom{\Le}{{\mathbb L}}
\ncom{\N}{{\mathbb N}}

\ncom{\p}{{\mathbb P}}
\ncom{\Q}{{\mathbb Q}}
\ncom{\R}{{\mathbb R}}
\ncom{\Z}{{\mathbb Z}}

\ncom{\f}{\dfrac}

\ncom{\wtil}{\widetilde}

\ncom{\ci}{{\mathpzc i}}

\ncom{\cA}{{\mathcal A}}
\ncom{\cC}{{\mathcal C}}
\ncom{\cE}{{\mathcal E}}
\ncom{\cF}{{\mathcal F}}
\ncom{\cG}{{\mathcal G}}
\ncom{\cH}{{\mathcal H}}
\ncom{\cI}{{\mathcal I}}
\ncom{\cJ}{{\mathcal J}}
\ncom{\cL}{{\mathcal L}}
\ncom{\cM}{{\mathcal M}}
\ncom{\cN}{{\mathcal N}}
\ncom{\cO}{{\mathcal O}}
\ncom{\cP}{{\mathcal P}}
\ncom{\cQ}{{\mathcal Q}}
\ncom{\cR}{{\mathcal R}}
\ncom{\cS}{{\mathcal S}}
\ncom{\cT}{{\mathcal T}}
\ncom{\cU}{{\mathcal U}}
\ncom{\cV}{{\mathcal V}}
\ncom{\cW}{{\mathcal W}}
\ncom{\cX}{{\mathcal X}}
\ncom{\cY}{{\mathcal Y}}
\ncom{\cZ}{{\mathcal Z}}

\ncom{\cSU}{{\mathcal S \mathcal U}}
\ncom{\eop}{{\hfill $\Box$}}
\ncom{\isom}{\cong}

\ncom{\todo}{{\textbf{TODO}}}

\newtheorem{theorem}{Theorem}[section]
\newtheorem{lemma}[theorem]{Lemma}

\newtheorem{proposition}[theorem]{Proposition}
\newtheorem{remark}[theorem]{Remark}

\newtheorem{question}[theorem]{Question}

\newtheorem{answer}[theorem]{Answer}

\begin{document}
\baselineskip=16pt

\title{Rank 3 arithmetically Cohen-Macaulay bundles over hypersurfaces}
\author{Amit Tripathi}
\address{Department of Mathematics, 
Indian Statistical Institute, Bangalore - 560059, India}
\email{amittr@gmail.com}
\thanks{This work was supported by a postdoctoral fellowship from NBHM, Department of Atomic Energy.}


\begin{abstract} Let $X$ be a smooth projective hypersurface of dimension $\geq 5$ and let $E$ be an arithmetically Cohen-Macaulay bundle on $X$ of any rank. We prove that $E$ splits as a direct sum of line bundles if and only if $H^i_*(X, \wedge^2 E) = 0$ for $i = 1,2,3,4$. As a corollary this result proves a conjecture of Buchweitz, Greuel and Schreyer for the case of rank 3 arithmetically Cohen-Macaulay bundles.
\end{abstract}
\maketitle
\section{Introduction}

We work over an algebraically closed field of characteristic 0. Let $\{X, \cO_X(1)\} \subset \p^{n+1}$ be a smooth projective hypersurface of degree $d$. We say a vector bundle on $X$ is \textit{split} if it can be written as a direct sum of line bundles. We say that it is \textit{indecomposable} if it can not be written as a direct sum of vector bundles of strictly smaller rank.

An \textit{arithmetically Cohen-Macaulay (ACM)}  vector bundle $E$ on $X$ is a locally free sheaf satisfying
$$
H^i_*(X, E) := \oplus_{k \in \Z}H^i(X, E(k)) = 0 \,\,\,\,\text{for} \,\,\, i = 1, \ldots n-1
$$

Some of the reasons why the study of ACM bundles is important are:
\begin{itemize}
\item On projective space, ACM bundles are precisely the bundles which are direct sum of line bundles \cite{Hor}.
\item By semicontinuity, ACM bundles form an open set in any flat family of vector bundles over $X$.
\item The $n$'th syzygy of a resolution of any vector bundle on $X$ by split bundles, is an arithmetically Cohen-Macaulay bundle \cite{Eisen81}.
\item These sheaves correspond to maximal Cohen-Macaulay modules over the associated coordinate ring \cite{Beau}.
\end{itemize}

When $d > 1$ there always exist indecomposable arithmetically Cohen-Macaulay bundles see e.g. \cite{M-R-R} for low dimensional construction and \cite{BGS} for a construction for higher dimensional hypersurfaces. The following conjecture forms the basis of research done in the direction of investigating the splitting behaviour of ACM bundles over hypersurfaces:

\begin{proof}[Conjecture (Buchweitz, Greuel and Schreyer \cite{BGS}):] Let $X \subset \p^n$ be a hypersurface. Let $E$ be an ACM bundle on $X$. If rank $E < 2^e$, where $e = \left [ \displaystyle{\frac{n-2}{2}} \right ]$, then $E$ splits. (Here $[q]$ denotes the largest integer $\leq q$.) 
\end{proof} 

This conjecture can not be strengthened further as the authors constructed an indecomposable ACM bundle of rank $2^e$ in \textit{op. cit.}   

For rank 2 ACM bundles, the conjecture follows from \cite{Kle}. Generic behaviour for rank 2 case is also well understood when $n \geq 4$ and we refer the reader to \cite{C-M1}, \cite{C-M2}, \cite{C-M3}, \cite{M-R-R}, \cite{M-R-R2}, \cite{R} and to the reference cited in these articles. For lower dimensional case, we refer the reader to \cite{Madonna1998}, \cite{Madonna2000}, \cite{Faenzi2008}, \cite{Chiantini-Faenzi2009} and \cite{C-H}. The result for rank 2 bundles was generalized to complete intersections in \cite{Biswas-Ravindra2010}.
 
For rank 3 ACM bundles the conjecture predicts splitting for $n \geq 5$ dimensional hypersurfaces. We proved a weaker version in \cite{Tripathi2015}. In this article, we prove the conjecture for rank 3 arithmetically Cohen-Macaulay bundles. 
\begin{theorem} \label{theorem_rank_3_splitting}
Let $X$ be a smooth hypersurface of dimension $\geq 5$. Let $E$ be a rank 3 arithmetically Cohen-Macaulay bundle over $X$. Then $E$ is a split bundle.
\end{theorem}

This result follows as a corollary from the main result of this article - a splitting criterion for ACM bundles of any rank.
\begin{theorem} \label{theorem_main_result}
Let $X$ be a smooth hypersurface of dimension $\geq 5$. Let $E$ be an arithmetically Cohen-Macaulay vector bundle on $X$ of any rank. Then $E$ splits if and only if $H^i_*(X,\wedge^2 E) = 0$ for $i = 1,2,3,4$.
\end{theorem}

%

\section{Preliminaries}

In this section, we will recall some standard facts about arithmetically Cohen-Macaulay bundles over hypersurfaces. 

Let $X \subset \p^{n+1}$ be a degree $d$ smooth hypersurface given by homogeneous polynomial $f = 0$. Let $E$ be an ACM bundle of rank $r$ on $X$. By Serre's duality, $E^{\vee}$ is also ACM. 

For notational ease, we will use $\,\, \wtil{} \,\,$ to denote a vector bundle on $\p^{n+1}$. By Hilbert's syzygy theorem, being a coherent sheaf on $\p^{n+1}$, $E$ admits a finite length minimal free resolution $$0 \rar \wtil{F_t} \rar \wtil{F_{t-1}} \rar \ldots \rar \wtil{F_1} \rar \wtil{F_0} \rar E \rar 0$$

where $\wtil{F_i}$ are direct sums of the form $\oplus_{j} \cO_{\p^{n+1}}(a_j)$. By minimality of the resolution and the ACM condition on $E$, the first syzygy $\wtil{K} = \text{Ker}(\wtil{F_0} \rar E)$ is an ACM bundle on $\p^{n+1}$ and therefore is a split bundle by Horrock's criterion. Thus the minimal free resolution of $E$ on $\p^{n+1}$ is of the form \begin{align}
\label{eqn_E_minimal_resolution_p} 0 \rar \wtil{F_1} \xrightarrow{\phi} \wtil{F_0} \rar E \rar 0
\end{align}

Localizing at the generic point, one checks that the ranks of $\wtil{F_1}$ and $\wtil{F_0}$ are same. Restricting the above resolution to $X$ gives,
$$
0 \rar Tor^1_{\p^{n+1}}(E, \cO_X) \rar \bar{F_1} \rar \bar{F_0} \rar E \rar 0
$$

where one computes the $Tor$ term by tensoring $0 \rar \cO_{\p^{n+1}}(-d) \xrightarrow{\times f} \cO_{\p^{n+1}} \rar \cO_X \rar 0$ with $E$ to get $Tor^1_{\p^{n+1}}(E, \cO_X) = E(-d)$ as multiplication by $f$ vanishes on $X$. Thus the above four term sequence breaks up as
\begin{align}  \label{eqn_E_1_step}
0 \rar E^{\sigma} \rar \bar{F_0} \rar E \rar 0
\end{align}
\begin{align}  \label{eqn_G_E}
0 \rar E(-d) \rar \bar{F_1} \rar E^{\sigma} \rar 0
\end{align}
where $\bar{F_i} = \wtil{F_i} \otimes \cO_X$ are split bundles over $X$ of rank $m$ and $E^{\sigma} := \text{Ker}(\bar{F_0} \onto E)$ is an arithmetically Cohen-Macaulay bundle on $X$.

We state the following facts (without proof) about matrix factorization theory of Eisenbud and the connection between $E$ and $E^{\sigma}$. We choose a matrix (with homogeneous polynomial entries) to represent the map $\phi: \wtil{F_1} \rar \wtil{F_0}$ and henceforth we will use the symbol $\phi$ interchangeably to represent either the matrix or the map. Then
\begin{enumerate}
\item There exists an injective map $\psi: \wtil{F_0}(-d) \rar \wtil{F_1}$ such that $\phi \psi = \psi \phi = f \mathbb{1}$ where $\mathbb{1}$ denotes the identity matrix.
\item Coker$(\psi) = E^{\sigma}$ and $E$ is indecomposable if and only if $E^{\sigma}$ is indecomposable.
\item $0 \rar \wtil{F_0}(-d) \rar \wtil{F_1} \rar E^{\sigma} \rar 0$ is a minimal free resolution of $E^{\sigma}$.
\end{enumerate}

For details, we refer to section 6 of \cite{Eisen81} and section 2 of \cite{C-H}.

\begin{lemma} \label{lemma_sheaf_resolution_1_length_is_reflexive}
Let $f$ be any homogeneous (perhaps reducible) polynomial of degree $d$. Let $X = V(f) \subset \p^{n+1}$ be the vanishing set. Suppose $\cF$ be any coherent sheaf on $X$ which admits a free resolution on $\p^{n+1}$ of the form $$0 \rar \wtil{F_1} \rar \wtil{F_0} \rar \cF \rar 0$$ where $\wtil{F_i}$ are direct sum of line bundles on $\p^{n+1}$. Then $\cF$ is a reflexive sheaf on $X$.
\end{lemma}
\begin{proof} We apply $\cH om(-, \cO_{\p^{n+1}})$ on the resolution of $\cF$ to get
$$
0 \rar \cH om(\cF, \cO_{\p^{n+1}}) \rar \wtil{F_0}^{\vee} \rar \wtil{F_1}^{\vee} \rar \cE xt^1(\cF, \cO_{\p^{n+1}}) \rar 0
$$

First term vanishes. To compute the $\cE xt$ term, we apply $\cH om(\cF, -)$ on $$0 \rar \cO_{\p^{n+1}}(-d) \rar \cO_{\p^{n+1}} \rar \cO_X \rar 0$$ to get $$0 \rar \cH om(\cF, \cO_{\p^{n+1}}) \rar \cH om(\cF, \cO_{X}) \rar  \cE xt^1(\cF, \cO_{\p^{n+1}})(-d) \xrightarrow{\times f}$$

Here the first term vanishes as before and the last map (multiplication by $f$) vanishes as the sheaves are supported on $X$. Thus we get $\cE xt^1(\cF, \cO_{\p^{n+1}}) \cong \cF^{\vee}(d)$ and a resolution of $\cF^{\vee}$ on $\p^{n+1}$ as
\begin{align} \label{equation_minimal_resolution_on_p_E_dual}
0 \rar \wtil{F_0}^{\vee}(-d) \rar \wtil{F_1}^{\vee}(-d) \rar \cF^{\vee} \rar 0
\end{align}

Applying the whole process once again to the above resolution of $\cF^{\vee}$ we get the following resolution of $\cF^{\vee \vee}$ $$0 \rar \wtil{F_1} \rar \wtil{F_0} \rar \cF^{\vee \vee} \rar 0$$

Comparing with the resolution of $\cF$, one gets the claim.
\end{proof}

Given a short exact sequence of vector bundles $0 \rar E_1 \rar E_2 \rar E_3 \rar 0$ on a variety $X$, there exists a resolution of the $k$'th exterior power $\wedge^k \,E_3$,
\begin{align} \label{equation_sym_wedge}
0 \rar S^k E_1 \rar S^{k-1} E_1 \otimes \wedge^1 \,E_2 \rar \ldots \wedge^{k} E_2 \rar \wedge^k \,{E_3} \rar 0
\end{align}

Dually, we also have a resolution of $k$'th symmetric power,
\begin{align} \label{equation_wedge_sym}
0 \rar \wedge^k \,E_1 \rar \wedge^k \,E_2 \rar \wedge^{k-1} \,E_2 \otimes S^1 E_3 \rar \ldots \wedge^1 E_2 \otimes S^{k-1} E_3 \rar S^{k} E_3 \rar 0
\end{align}

For details we refer the reader to \cite{Buchsbaum-Eisen1975}.
\section{A cokernel sheaf}

Suppose $\text{rank}\,\wtil{F_0} = \text{rank}\,\wtil{F_1} = m$. Fix any integer $k \leq \text{min}\{\text{rank}(E), \text{rank}(E^{\sigma})\}$. Let $X_k = V(f^k)$ denote the scheme-theoretic $k$'th thickening of $X \subset \p^{n+1}$.

We consider the $k$'th exterior power of the map $\phi:\wtil{F_1} \rar \wtil{F_0}$ in equation \eqref{eqn_E_minimal_resolution_p} and denote the cokernel sheaf by  $\cF_k$
\begin{align} \label{equation_k_exterior_power}
0 \rar \wedge^k  \wtil{F_1} \xrightarrow{\wedge^k \phi} \wedge^k  \wtil{F_0} \rar \cF_k \rar 0
\end{align}

The following lemma states some properties of the sheaf $\cF_k$. Our proof is similar to that in section 2 of \cite{M-R-R} where the case when $E$ is a rank 2 ACM bundle and $k = 2$ was studied. 

\begin{lemma} \label{lemma_cF_properties} 
\begin{enumerate}
\item $\cF_k$ is a coherent $\cO_{X_k}$-module where $X_k$ is the thickened hypersurface defined scheme theoretically by $f^k$.
\item $\bar{\cF_k} := \cF_k \otimes \cO_X$ is a vector bundle on $X$ of rank $\binom{m}{k} - \binom{m-r}{k}$
\item $\cF_k$ is an ACM and reflexive sheaf on $X_k$.
\end{enumerate}  
\end{lemma}
\begin{proof} First two claims can be verified locally. By localising on $X$, one can assume that equation \eqref{eqn_E_minimal_resolution_p} looks like $$ 0 \rar \cO_p^{\oplus m} \xrightarrow{\phi} \cO_p^{\oplus m} \rar E_p \rar 0$$ and the matrix $\phi$ is given by the $m \times m$ diagonal matrix $$\{f,\ldots\ldots, f, 1, \ldots ,1\}$$ where $f$ appears $r = \text{rank}(E)$ times and $1$ appears $m-r$ times on the diagonal. Then locally the matrix $\wedge^k \, \phi$ is the diagonal matrix $$\{f^k, \ldots f^k, f^{k-1} \ldots f^{k-1}, f^{k-2}, \ldots \ldots f, 1,1,\ldots 1\}$$

where $f^{k-i}$ appears $\binom{r}{k-i}\binom{m-r}{i}$ times on the diagonal. In particular, locally $\cF_k$ is of the form $$\cO_{X_k}^{\oplus \binom{r}{k}} \oplus \cO_{X_{k-1}}^{\oplus \binom{r}{k-1} \cdot \binom{m-r}{1}} \oplus \ldots \oplus \cO_{X_{k-i}}^{\oplus \binom{r}{k-i} \cdot \binom{m-r}{i}} \ldots \oplus \cO_{X}^{\oplus \binom{r}{1} \cdot \binom{m-r}{k-1}}$$

This proves the first claim and also that $\bar{\cF_k} = \cF_k \otimes \cO_X$ is a vector bundle on $X$. Claim about the rank is verified by the above local description of $\cF_k$ and the combinatorial identity $$\binom{m}{k} = \sum_i \binom{r}{i} \binom{m-r}{k-i}$$

By equation \eqref{equation_k_exterior_power}, one easily sees that $\cF_k$ is an ACM sheaf on $X_k$. Lemma \ref{lemma_sheaf_resolution_1_length_is_reflexive} completes the proof by showing that $\cF_k$ is a reflexive sheaf.
\end{proof}

We now restrict sequence \eqref{equation_k_exterior_power} to $X$
\begin{align} \label{equation_k_ext_pow_rest_X}
0 \rar Tor^1_{\p^{n+1}}(\cF_k, \cO_X) \rar \wedge^k  \bar{F_1} \rar \wedge^k  \bar{F_0} \rar \bar{\cF_k} \rar 0
\end{align}

This is a sequence of vector bundles and the $Tor$ term is a vector bundle of same rank as $\bar{\cF_k}$. In fact, the map $F_1 \rar F_0$ factors via $E^{\sigma}$, therefore by functoriality of exterior product, the map  $\wedge^k  \,\bar{F_1} \rar \wedge^k  \,\bar{F_0}$ factors via $\wedge^k \, E^{\sigma}$ and the sequence \eqref{equation_k_ext_pow_rest_X} breaks up as
\begin{align} \label{equation_k_ext_pow_rest_X_Tor}
0 \rar Tor^1_{\p^{n+1}}(\cF_k, \cO_X) \rar \wedge^k  \bar{F_1} \rar \wedge^k \,E^{\sigma} \rar 0 
\end{align}

and
\begin{align} \label{equation_k_ext_pow_rest_X_G}
 0 \rar \wedge^k \,E^{\sigma} \rar \wedge^k  \bar{F_0} \rar \bar{\cF_k} \rar 0
\end{align}

Thus the $Tor$ term appears as the first term in the filtration of $k$'th exterior power of $\bar{F_1}$ derived from the sequence $0 \rar E(-d) \rar \bar{F_1} \rar E^{\sigma} \rar 0$. We can say more,
\begin{lemma}
$Tor^1_{\p^{n+1}}(\cF_k, \cO_X) \cong \overline{\cF_k^{\vee}}^{\vee}(-kd)$
\end{lemma}
\begin{proof} We consider the $k$'th exterior power of the minimal resolution of $E^{\vee}$ given by sequence \eqref{equation_minimal_resolution_on_p_E_dual}
\begin{align} \label{equation_defining_F'}
0 \rar (\wedge^k \wtil{F_0}^{\vee})(-kd) {\rar} (\wedge^k \wtil{F_1}^{\vee})(-kd) \rar \cF'_k \rar 0
\end{align}

where $\cF'_k$ is defined by the sequence. Restricting to $X$ gives
$$
0 \rar Tor^1_{\p^{n+1}}(\cF'_k,\cO_X) \rar (\wedge^k \bar{F_0}^{\vee})(-kd) {\rar} (\wedge^k \bar{F_1}^{\vee})(-kd) \rar \bar{\cF'_k} \rar 0
$$

As in lemma \ref{lemma_cF_properties} one can verify (by looking at the exterior power matrix locally) that $\bar{\cF}'_k$ is a vector bundle and thus above is a exact sequence of vector bundles. So we can dualize (and then twist by $-kd$) to get:
\begin{align} 
0 \rar \bar{\cF'}^{\vee}_k(-kd)  \rar \wedge^k {\bar{F_1}} {\rar} \wedge^k {\bar{F_0}} \rar Tor^1(\cO_X, \cF'_k)^{\vee}(-kd) \rar 0
\end{align}

Comparing with equation \eqref{equation_k_ext_pow_rest_X}, we get 
\begin{align}
Tor^1(\cF_k, \cO_X, ) \cong \bar{\cF'}^{\vee}_k(-kd)
\end{align}

We complete the proof by showing that $\cF'_k \cong \cF^{\vee}_k$. Applying $\cH om(-, \cO_{\p^{n+1}})$ to sequence \eqref{equation_defining_F'} and simplifying as in the proof of Lemma \ref{lemma_sheaf_resolution_1_length_is_reflexive}, we get 
\begin{align} \label{sequence_F'_wedge_dual} 0 \rar \wedge^k \wtil{F_1} \rar \wedge^k \wtil{F_0} \rar \cF'^{\vee}_k \rar 0 
\end{align} 

Comparing this with the sequence \eqref{equation_k_exterior_power} and using the fact that by Lemma \ref{lemma_sheaf_resolution_1_length_is_reflexive}, $\cF_k, \cF'_k$ are both reflexive sheaves, we get that $\cF^{\vee}_k \cong \cF'_k$. 
\end{proof}

\begin{lemma} \label{lemma_wedge_k_E_kernel_Tor_terms}
There exists a short exact sequence $$0 \rar \wedge^k E(-kd) \rar Tor^1_{\p}(\cF_k, \cO_X) \rar Tor^1_{X_k}(\cF_k, \cO_X) \rar  0$$
\end{lemma}
\begin{proof} We restrict the sequence \eqref{equation_k_exterior_power} to $X_k$ to get a free $\cO_{X_k}$-resolution of $\cF_k$
$$
\cdots \rar \wedge^k F_1(-kd) \rar \wedge^k F_0(-kd) \rar \wedge^k F_1 \rar \wedge^k F_0 \rar \cF_k \rar 0
$$

Tensoring this resolution with $\cO_X$ gives a complex from which we get 
\begin{align} \label{equation_Tor_X_k_equals_mod_of_wedge_maps}
Tor^1_{X_k}(\cF_k, \cO_X) \cong \dfrac{\text{Ker}(\wedge^k \bar{F_1} \rar \wedge^k \bar{F_0})}{\text{Im}(\wedge^k \bar{F_0}(-kd) \rar \wedge^k \bar{F_1})}
\end{align}

To compute $\text{Ker}(\wedge^k \bar{F_1} \rar \wedge^k \bar{F_0})$, we tensor the sequence \eqref{equation_k_exterior_power} with $\cO_X$ to get $$\text{Ker}(\wedge^k \bar{F_1} \rar \wedge^k \bar{F_0}) \cong Tor^1_{\p}(\cF_k, \cO_X)$$

For the $\text{Im}(\wedge^k \bar{F_0}(-kd) \rar \wedge^k \bar{F_1})$ term, we note that the map $\bar{F_0}(-d) \rar \bar{F_1}$ factors via $E(-d)$ so by functoriality of wedge power, $$\text{Im}(\wedge^k \bar{F_0}(-kd) \rar \wedge^k \bar{F_1}) \cong \wedge^k E(-kd)$$

This completes the proof of the lemma.
\end{proof}
\subsection{A short exact sequence}

Let $\cF$ be any coherent $\cO_{X_k}$-module. The inclusions $X_{k-1} \into \p^{n+1}$ and $X \into X_k$ induces following short exact sequences
\begin{align} \label{sequence_X_into_X_k}
0 \rar \cO_{X_{k-1}}(-d) \rar \cO_{X_k} \rar \cO_X \rar 0
\end{align}
\begin{align} \label{sequence_X_k-minus-1_into_X_k}
0 \rar \cO_{\p}(-(k-1)d) \rar \cO_{\p} \rar \cO_{X_{k-1}} \rar 0
\end{align}

Tensoring both sequences with $\otimes_{\p} \cF$, we get 
\begin{align} \label{sequence_X_into_X_k_tensored_cF_on_p}
0 \rar Tor^1_{\p}(\cF, \cO_{X_{k-1}}(-d)) \rar \cF(-kd) \rar Tor^1_{\p}(\cF, \cO_X) \rar \cF|_{X_{k-1}}(-d) \rar \cF\rar \overline{\cF} \rar 0
\end{align}

\begin{align} \label{sequence_X_k_1_into_p_tensored_cF_on_p}
0 \rar Tor^1_{\p}(\cF, \cO_{X_{k-1}}) \rar \cF(-(k-1)d) \rar \cF \rar \cF|_{X_{k-1}} \rar 0
\end{align}

Similarly, tensoring sequence \eqref{sequence_X_into_X_k} with  $\otimes_{X_k} \cF$, we get
\begin{align} \label{sequence_X_into_X_k_tensored_cF_on_X_k}
0 \rar Tor^1_{X_k}(\cF, \cO_X) \rar \cF|_{X_{k-1}}(-d) \rar \cF\rar \overline{\cF} \rar 0
\end{align}

Comparing sequences \eqref{sequence_X_into_X_k_tensored_cF_on_p} and \eqref{sequence_X_into_X_k_tensored_cF_on_X_k} gives 
\begin{align} \label{sequence_main_sequence_one}
0 \rar Tor^1_{\p}(\cF, \cO_{X_{k-1}})(-d) \rar \cF(-kd) \rar Tor^1_{\p}(\cF, \cO_X) \rar Tor^1_{X_k}(\cF, \cO_X) \rar 0
\end{align}

\begin{lemma} \label{lemma_kernels_are_isomorphic}
With notations as above, $$\text{Ker}[Tor^1_{\p}(\cF, \cO_X) \onto Tor^1_{X_k}(\cF, \cO_X)] \cong \text{Ker}[\cF(-d) \onto \cF|_{X_{k-1}}(-d)]$$
\end{lemma}
\begin{proof} Twist the sequence \eqref{sequence_X_k_1_into_p_tensored_cF_on_p} by $-d$ and compare it with the sequence \eqref{sequence_main_sequence_one}.
\end{proof}

\begin{proposition} \label{prop_sequence_wedge_k_k-minus-1}
There exists a short exact sequence $$0 \rar \wedge^k E(-(k-1)d) \rar \cF_k \rar \cF_k|_{X_{k-1}} \rar 0$$
\end{proposition}
\begin{proof} Follows from Lemma \ref{lemma_wedge_k_E_kernel_Tor_terms} and by putting $\cF = \cF_k$ in Lemma \ref{lemma_kernels_are_isomorphic}.
\end{proof}

\section{Proof of the theorem}

We now apply above results for $k = 2$. 

\begin{proposition} \label{lemma_wedge_2_E_ACM_iff_wedge_2_G_ACM}
Let $E$ be an ACM bundle on a smooth hypersurface of dimension $\geq 3$. Then $\wedge^2 \,E$ is ACM if and only if $\wedge^2 \,E^{\sigma}$ is ACM.
\end{proposition}
\begin{proof} Assume that $\wedge^2 \,E$ is ACM. For $k = 2$, we get following short exact sequences for $E$ (sequence \eqref{equation_k_ext_pow_rest_X_G} and the sequence from Lemma \ref{prop_sequence_wedge_k_k-minus-1})
\begin{align} \label{sequence_main_one}
0 \rar \wedge^2 \,E^{\sigma} \rar \wedge^2 \bar{F_0} \rar \bar{\cF_2} \rar 0
\end{align} 
\begin{align} \label{sequence_main_two}
0 \rar \wedge^2 \,E(-d) \rar \cF_2 \rar \bar{\cF_2} \rar 0
\end{align}

Comparing sequences \eqref{sequence_main_one}, \eqref{sequence_main_two} and using the fact that $\wedge^2 \bar{F_0}, \cF_2$ are all ACM, we get $H^i_*(\wedge^2 \, E^{\sigma}) = 0$ when $i = 2, \ldots n-1$ where $n = \text{dim}(X)$. 

To prove the vanishing for $i = 1$, we note that $E^{\vee}$ is also ACM and $E^{\vee\sigma} \cong E^{\sigma\vee}(-d)$, e.g. by lemma 2.5 of \cite{C-H}. Therefore the same proof shows that $H^i_*(\wedge^2 \, (E^{\sigma \,\vee})) = 0$ when $i = 2, \ldots n-1$. Applying Serre's duality completes the proof.
\end{proof}

We now prove our main result,

\begin{proof}[Proof of Theorem \ref{theorem_main_result}] Suffices to show one direction. Assume $H^i_*(X,\wedge^2 \,E) = 0$ for $i = 1,2,3,4$. Consider the composition of sequences \eqref{equation_sym_wedge} and \eqref{equation_wedge_sym}:
$$
0 \rar \wedge^2 \,E(-2d) \rar \wedge^2 \bar{F_1} \rar \bar{F_1} \otimes E^{\sigma} \rar E^{\sigma} \otimes \bar{F_0} \rar \wedge^2 \bar{F_0} \rar \wedge^2 E \rar 0
$$

One concludes that $H^i(X, \wedge^2 E(k)) = H^{i+4}(X, \wedge^2 E(k-2d))$ for $i = 1, \ldots n-5$. Thus $\wedge^2 \,E$ is ACM. By Lemma \ref{lemma_wedge_2_E_ACM_iff_wedge_2_G_ACM}, $\wedge^2 \,E^{\sigma}$ is also ACM. We consider sequence \eqref{equation_sym_wedge}
$$
0 \rar S^2 E(-d) \rar E(-d) \otimes \bar{F_1} \rar \wedge^2 \bar{F_1} \rar \wedge^2 E^{\sigma} \rar 0
$$

This gives $H^i_*(S^2 E) = 0$ when $i = 3, \ldots n-1$. Since $\wedge^2 E$ is ACM implies $\wedge^2 E^{\vee}$ is also ACM, we do a dual analysis to get $H^i_*(S^2 E^{\vee}) = 0$ when $i = 3, \ldots n-1$. Applying Serre's duality and combining this with the vanishing for $S^2 E$, we get that when $n - 3 \geq 2$ then $S^2 E$ is also ACM. 

Thus when $\text{dim} (X) \geq 5$, $E \otimes E = \wedge^2 E \oplus S^2 E$ is ACM which by Theorem \ref{theorem_corollary_weigand_huneke} implies that $E$ is split.
\end{proof}


\begin{remark}
We note that the statement $\wedge^2 \, E$ is ACM implies $E \otimes E$ is ACM is tight in the dimension. For a counterexample in lower dimension, consider any rank 2 indecomposable ACM vector bundle on a hypersurface of dimension 4. Then $\wedge^2 \, E$ is ACM but $E \otimes E \cong E \otimes E^{\vee}(t)$ can not be ACM for otherwise $H^2_*(X, \cE nd(E)) = 0$ and hence in particular, by lemma 2.2 of \cite{M-R-R}, $E$ is split which contradicts the indecomposability of $E$. 
\end{remark}

\section{$E \otimes E$ is ACM implies $E$ is split}

Let $f \in R = k[x_0, x_1, \ldots x_{n+1}]$ be a homogeneous irreducible polynomial of positive degree. Let $S = R/(f)$ and $X = Proj(S)$ be the corresponding hypersurface.

We state the following result without proof
\begin{lemma}
Let $E$ be a vector bundle on $X$. Let $M = H^0_*(X, E)$ be corresponding graded $S$-module. Then $E$ splits if $M$ is a free $S$-module.
\end{lemma}

Following result is Theorem 3.1 in \cite{H-W}
\begin{theorem}[Huneke-Weigand] \label{theorem_of_huneke_weigand}
Let $(R,m)$ be an abstract hypersurface and let $M, N$ be $R$-modules, at least one of which has constant rank. If $M \otimes_R N$ is a maximal Cohen-Macaulay $R$-module then either $M$ or $N$ is free. 
\end{theorem}

The corresponding version for vector bundles is of course not true as every vector bundle on a planar curve is ACM (vacuously) and there exists indecomposable vector bundles on various planar curves. Though for our need, the following corollary suffices.
\begin{theorem}[Corollary to Theorem \ref{theorem_of_huneke_weigand}] \label{theorem_corollary_weigand_huneke} Let $X = \text{Proj }(S)$ be a hypersurface of dimension $\geq 3$. Let $E$ be an ACM vector bundle on $X$. Further assume that $E \otimes E$ is ACM. Then $E$ splits. 
\end{theorem}
\begin{proof} We consider a minimal resolution of $E$ on $X$
\begin{align} \label{sequence_E_1_step_X}
0 \rar E^{\sigma} \rar \bar{F_0} \rar E \rar 0
\end{align}
and 
\begin{align} \label{sequence_G_1_step_X}
0 \rar  E(-d) \rar \bar{F_1} \rar E^{\sigma} \rar 0
\end{align}

Where $\bar{F_0}, \bar{F_1}$ are direct sum of line bundles. Tensoring sequence \eqref{sequence_E_1_step_X} with $E$ and sequence \eqref{sequence_G_1_step_X} with $E^{\sigma}$ and using the fact that $E \otimes E$ is ACM, we deduce that $E \otimes E^{\sigma}$ is ACM. Thus there exists a short exact sequence of graded $S$-modules:
\begin{align*} \label{sequence_H_0_graded_E_tensor_G}
0 \rar  H^0_*(E^{\sigma} \otimes E) \rar H^0_*(\bar{F_0} \otimes E) \rar H^0_*(E \otimes E) \rar 0 
\end{align*}

Here we are using the fact that $\text{dim}(X) \geq 3$. Sequence \eqref{sequence_E_1_step_X} yields the following right exact sequence
\begin{align*}
H^0_*(E^{\sigma}) \otimes H^0_*(E) \rar H^0_*(\bar{F_0}) \otimes H^0_*(E) \rar H^0_*(E) \otimes H^0_*(E) \rar 0
\end{align*}

Thus we get the following commutative diagram

\vspace{2mm}

\xymatrix{&H^0_*(E^{\sigma}) \otimes H^0_*(E) \ar[r] \ar[d]^{\phi_2} &H^0_*(\bar{F_0}) \otimes H^0_*(E) \ar[r] \ar@{=}[d] &H^0_*(E) \otimes H^0_*(E) \ar[r] \ar[d]^{\phi_1} &0 \\ 0 \ar[r]  &H^0_*(E^{\sigma} \otimes E) \ar[r] &H^0_*(\bar{F_0} \otimes E) \ar[r] &H^0_*(E \otimes E) \ar[r] &0}

where the all vertical maps are naturally defined. Middle map is an equality because $\bar{F_0}$ is a split bundle. By Snake's lemma, $\phi_1$ is a surjective map.

Similarly we get following commutative diagram from the sequence \eqref{sequence_G_1_step_X}
 
 \xymatrix{ &H^0_*(E(-d)) \otimes H^0_*(E) \ar[r] \ar[d] &H^0_*(\bar{F_1}) \otimes H^0_*(E) \ar[r] \ar@{=}[d] &H^0_*(E^{\sigma}) \otimes H^0_*(E) \ar[r] \ar[d]^{\phi_2} &0 \\ 0 \ar[r]  &H^0_*(E(-d) \otimes E) \ar[r] &H^0_*(\bar{F_1} \otimes E) \ar[r] &H^0_*(E^{\sigma} \otimes E) \ar[r] &0}

By Snake's lemma $\phi_2$ is surjective. In turn this implies that $\phi_1$ is injective and hence $H^0_*(E) \otimes H^0_*(E) \rar H^0_*(E \otimes E)$ is an isomorphism. Thus $H^0_*(E) \otimes H^0_*(E)$ is a maximal Cohen-Macaulay module and we can apply Theorem \ref{theorem_of_huneke_weigand} to conclude that $H^0_*(E)$ is free and therefore $E$ splits.  
\end{proof}

\begin{proof}[Proof of Theorem \ref{theorem_rank_3_splitting}] The perfect pairing $E \times \wedge^2 \, E \mapsto \wedge^3 \, E = \cO_X(e)$ induces an isomorphism $\wedge^2 \, E \cong E^{\vee}(e)$. By Serre's duality then $\wedge^2 \, E$ is ACM and hence we can apply Theorem \ref{theorem_main_result}.
\end{proof}

%

\section{Acknowledgement}

We thank Suresh Nayak for pointing out a crucial mistake in an earlier version. We thank Jishnu Biswas and Girivaru Ravindra for constant support and useful conversations.


\end{document}